\documentclass[reqno]{amsart}
\usepackage[T1]{fontenc}

\usepackage{amssymb, amsthm} 
\usepackage{mathrsfs}

\newtheorem{theorem}{Theorem}
\newtheorem{lemma}{Lemma}
\newtheorem{proposition}{Proposition}

\usepackage[colorlinks,linkcolor=black,urlcolor=blue,citecolor=black]{hyperref} 

\usepackage{enumitem}

\newcommand{\floor}[1]{\lfloor #1 \rfloor}

\everymath{\displaystyle}

\title{Primitive Divisors of Lucas Sequences in Polynomial Rings}
\author{Joaquim Cera Da Concei\c c\~ao}
\address{Normandie Universit\'e, UNICAEN, CNRS, LMNO, 14000 Caen, France}
\email{\tt joaquim.cera-daconceicao@unicaen.fr}
\urladdr{\href{https://jceradaconceicao.github.io}{\tt https://jceradaconceicao.github.io}}


\begin{document}

\begin{abstract}
It is known that all terms $U_n$ of a classical regular Lucas sequence have a primitive prime
divisor if $n>30$ \cite{BHV}. In addition, a complete description of all regular Lucas sequences and their
terms $U_n$, $2\leq n\leq 30$, which do not have a primitive divisor is also known. Here, we prove 
comparable results for Lucas sequences in polynomial rings, correcting some previous theorem
on the same subject. The first part of our paper develops some elements of Lucas theory in
several abstract settings before proving our main theorem in polynomial rings.
\end{abstract}

\maketitle
\tableofcontents

\section{Introduction}

Let $a$ and $b$ be non-zero integers such that $a\ne \pm b$. In 1892, Zsigmondy \cite{Zsi} proved that for all $n\geq 1$ but a few exceptions, there exists a prime number that divides $a^n-b^n$ but not any $a^k-b^k$, for $1\leq k <n$. Such a prime is called a primitive divisor of $a^n-b^n$. In particular, there always exists a primitive prime divisor for $n\geq 7$. Knowing that $a-b$ divides $a^n-b^n$ for all $n\geq 1$, numbers of the form
\begin{equation}\label{EqGen}
    \frac{a^n-b^n}{a-b},
\end{equation}
also have a primitive divisor for $n\geq 7$. Moreover, such a formula defines a second order linear recurrence known as a Lucas sequence. Indeed, a Lucas sequence $U=(U_n)_{n\geq 0}$ with parameters $P,Q\in\mathbb{Z}$ is a sequence with initial terms $U_0=0$ and $U_1=1$ satisfying
\[
U_{n+2}=PU_{n+1}-QU_n,
\]
for all $n\geq 0$. The polynomial $f(X)=X^2-PX+Q$ is the characteristic polynomial of $U$. We denote its discriminant $P^2-4Q$ by $\Delta$. If $a$ and $b$ are the roots of $f$ in $\mathbb{Q}(\sqrt{\Delta})$, then either
\[
U_n=na^{n-1} \quad \text{or} \quad U_n=\frac{a^n-b^n}{a-b},
\]
for all $n\geq 0$, depending on whether $a= b$ or $a\ne b$ respectively. Therefore, \eqref{EqGen} leads to another generalization of Zsigmondy's theorem with $a$ and $b$ not only integers, but also quadratic conjugates: is there an $n_0\geq 1$ such that $U_n$ has a primitive prime divisor $p\nmid \Delta = (a-b)^2$ for all $n\geq n_0$? If $n_0$ exists, can we classify all sequences $U$ and integers $1\leq n\leq n_0$ for which $U_n$ fails to have a primitive divisor?

A partial answer was given by Carmichael \cite{Ca} in the case $\gcd(P,Q)=1$ and $\Delta> 0$, where $n_0=13$. For instance, let $P=1$ and $Q=-1$. Then $U$ is the Fibonacci sequence $F$ which has the form
\[
F_n = \frac{\phi^n-\Bar{\phi}^n}{\sqrt{5}},
\]
for all $n\geq 0$, where $\phi$ is the golden ratio and $\Bar{\phi}$ its quadratic conjugate. Carmichael showed that $F$ has a primitive prime divisor that does not divide $\Delta$ for all $n\not \in \{1,2,5,6,12\}$. A full answer was given by Bilu, Hanrot, and Voutier \cite{BHV} when $P$ and $Q$ are relatively prime. Their theorem states that $U$ has a primitive prime divisor $p\nmid \Delta$ for all $n\geq 31$. Moreover, they give a full description of the sequences $U$ and indices $1\leq n \leq 30$ for which the theorem fails. Note that all such sequences, except for one, satisfy Carmichael's bound $n_0=13$. Indeed, the only exception is the Lucas sequence with parameters $P=1$ and $Q=2$ for which $U_{30}$ does not have a primitive divisor.

Further generalizations may be considered. Here, we are concerned with polynomial rings $A[T]$ with $A$ a unique factorization domain. This makes $A[T]$ a unique factorization domain in which the irreducible elements are irreducible polynomials. In this setting, we call a {\it primitive prime divisor} of $U_n$ an irreducible polynomial that divides $U_n$, but no $U_k$ for all $1\leq k \leq n$, for $U$ a Lucas sequence. Recently, Sha \cite{Sha} investigated the case of rings $A$ of multivariate polynomials over an arbitrary field $K$ of characteristic $p$, where $p$ is zero or a prime number. In particular, Sha studied the case of Lucas sequences with coprime parameters $P,Q\in A$ and proved the following theorem, see \cite[Theorem 1.4]{Sha}, which we state as Sha did:
\begin{theorem}\label{Sha's Theorem}
    Suppose the characteristic $p > 0$ and let $U'$ be the sequence obtained from $(U_n)_{n\geq 1} $ by deleting the terms $U_n$ with $p\mid n$, then each term of $U'$ beyond the second has a primitive prime divisor. If $p = 0$, then each term of $(U_n)_{n\geq 1} $ beyond the second has a primitive prime divisor.
\end{theorem}

However, there seems to be a small mistake towards the end of Sha's proof, invalidating his result. Indeed, his proof relies on the fact that a term $U_n$, for $n\geq 2$, has a primitive divisor if and only if $\Phi_n(a,b)$ has a positive polynomial degree. Here, $a$ and $b$ are the roots of the characteristic polynomial of $U$, and $\Phi_n$ is the $n$-th homogeneous cyclotomic polynomial. In his proof, Sha states that $\Phi_n(a,b)$ must have positive degree, since at least one of $a$ and $b$ is transcendental over the field $K$, and does not give further explanation. Note that this assertion is used in the proof of a similar theorem for {\it Lehmer sequences} \cite[Theorem 1.2]{Sha}, and that Sha deduces another result \cite[Theorem 1.6]{Sha} from Theorem \ref{Sha's Theorem}. Let us give a few counterexamples to this statement and thus to Theorem \ref{Sha's Theorem}. Let $A$ be any multivariate polynomial ring over a field $K$ with characteristic $p\not \in \{2,3\}$, $P\in A$ be non-constant and $\lambda\in K^\times$. Let $U$ be the Lucas sequence with characteristic polynomial
\[
f(X) = X^2-PX+(P^2-\lambda).
\]
Then, the first terms of $U$ are $U_0=0, U_1=1, U_2=P$ and $U_3=\lambda$. Since $\lambda$ is a non-zero constant, we see that $U_3$ does not have a primitive prime divisor. Now, the roots of $f$ are
\[
a=\frac{P- \sqrt{4\lambda-3P^2}}{2} \quad \text{and} \quad b=\frac{P + \sqrt{4\lambda-3P^2}}{2},
\]
which are transcendental over $K$ because $P$ is not a constant. Hence, we obtain $\Phi_3(a,b) = a^2+ab+b^2 = \lambda$, contradicting Sha's statement. For a more specific example, we let $A=\mathbb{F}_q[T]$, where $\mathbb{F}_q$ is the finite field with $q=7^s$ elements, $s\geq 1$. Define $U$ with characteristic polynomial
\[
f(X) = X^2-4TX+(3T^2-1).
\]
We find that $U_6=3U_2U_3$. Thus, $U_6$ has no primitive prime divisor. The roots of $f$ are $a=2T-\sqrt{T^2+1}$ and $b=2T+\sqrt{T^2+1}$, and $\Phi_6(a,b) = 3$.

The object of this paper is to obtain a corrected version of Sha's theorem for $A[T]$, where $A$ is a unique factorization domain. We use the same approach as Yabuta \cite{Ya}, who gave a simplified proof of Carmichael's result.

Section \ref{Section2} is split in two parts. We first define and prove various divisibility properties of Lucas sequences in integral domains. The proofs given are similar, if not the same, as the ones given for the usual Lucas sequences in $\mathbb{Z}$. See \cite[Theorems 38 and 2, Lemma 2, Theorem 3, Lemmas 8 and 7, and Theorem 8]{BaWi} for the $\mathbb{Z}$-analogues of Lemmas \ref{LemmaDegenerate}, \ref{LemmaDivSeq}, and \ref{LemmaQUn=R}, Proposition \ref{PropStrongDivSeq}, Lemmas \ref{LemmaRanknmidQ} and \ref{LemmaUnUn+1ACCP}, and Proposition \ref{PropPdivUrhodivn} respectively. Then, we study prime ideals that divide terms of a Lucas sequence in a unique factorization domain.

We prove our main theorem in Section \ref{Section3}. For a Lucas sequence $U$ and an integer $n\geq 1$, we show that if $p\nmid n$, then $U_n$ has a primitive prime divisor except for $n=1$ and possibly at most one value in $\{2,3,4,6\}$. For each $n$ in $\{2,3,4,6\}$, we give a condition on $P$ and $Q$ for $U_n$ to not have a primitive prime divisor.

Throughout this paper, the letters $n$ and $p$ denote respectively an integer, and zero or a prime number. For an arbitrary field $L$, we let $\Bar{L}$ denote an algebraic closure of $L$ and $\zeta_n \in \Bar{L}$ denote a primitive $n$-th root of unity. For $x$ an element of a ring, we write $(x)$ for the principal ideal generated by $x$. The notation $(m,n)$ is used to denote the greatest common divisor of integers $m$ and $n$.

\section{Basic properties of general Lucas sequences}\label{Section2}

Throughout this section, the letter $\mathscr{R}$ denotes an integral domain with characteristic $p$. Let $\mathscr{K}$ be the fraction field of $\mathscr{R}$. For non-zero $P,Q\in \mathscr{R}$, we consider $f(X)= X^2-PX+Q$ with roots $a,b\in \Bar{\mathscr{K}}$. If $p\ne 2$, then by the discriminant method, we have $\Delta=P^2-4Q$ and 
\[
a=\frac{P+\sqrt{\Delta}}{2} \quad \text{and} \quad b=\frac{P-\sqrt{\Delta}}{2}.
\]
This method does not work in characteristic $2$. However, if $a$ is a root of $f$ then $b=a+P$ is the other root of $f$. Putting $\Delta=P^2-4Q=P^2$ in this case, we see that $a-b=\sqrt{\Delta}$ for all values of $p$. Let $U=U(P,Q)$ be the sequence defined by $U_0=0$, $U_1=1$ and $U_{n+2}=PU_{n+1}-QU_n$ for all integers $n\geq 0$. We see that $U$ has characteristic polynomial $f$ and, with $a$ and $b$ the roots of $f$, we find the explicit classic formula for $U_n$ to be
\[
U_n=\frac{a^n-b^n}{a-b},
\]
or $U_n=na^{n-1}$ for all $n\geq 0$, depending on whether $a\ne b$ or $a=b$ respectively. We call $U$ the {\it Lucas sequence} with parameters $P,Q\in \mathscr{R}$. 

We split this section into two subsections. We first give properties related to Lucas sequences in $\mathscr{R}$. Next, we study the behavior of prime ideals in a Lucas sequence in the case of a unique factorization domain.

\subsection{Integral domains}

We consider two types of Lucas sequences: the degenerate and non-degenerate Lucas sequences. A sequence $U$ such that $U_n=0$ for some positive $n$ is said to be degenerate. Therefore, a Lucas sequence which is not zero for all $n\geq 1$ is called non-degenerate. The following result shows that degeneracy can be expressed as a relation between the roots of the characteristic polynomial:

\begin{lemma}\label{LemmaDegenerate}
    A Lucas sequence $U$ is degenerate if and only if $a=\zeta b$ for some root of unity $\zeta\in \mathscr{K}(a)$, where $\zeta\ne 1$ if $p=0$.
\end{lemma}

\begin{proof}
    If $a=b$, then $U_n=na^{n-1}=0$ if and only if $p\mid n$ because $Q\ne 0$. If $a\ne b$, then $U_n=0$ for some $n\geq 1$ if and only if
    \[
    \frac{a^n-b^n}{a-b} =0 \iff \bigg(\frac{a}{b}\bigg)^n = 1 \iff a=\zeta_n b,
    \]
    where $\zeta_n$ is an $n$-th root of unity in $\mathscr{K}(a)$.
\end{proof}

Clearly, if $U_n$ is zero for some $n\geq 1$, it follows that $U$ cannot have primitive divisors past that term. Therefore, we assume throughout this paper that $U$ is non-degenerate.

\begin{proposition}\label{PropUm+n}
    For all $m>n\geq 0$, we have $U_m = U_{n+1}U_{m-n} - QU_n U_{m-n-1}$.
\end{proposition}

\begin{proof}
    Induction on $m\geq n+1$ for a fixed $n\geq 0$.
\end{proof}

Note that this proposition is valid for Lucas sequences in any commutative ring and is widely used throughout this section. Another important tool is the next lemma, which states that $U$ is a divisibility sequence:

\begin{lemma}[Divisibility sequence]\label{LemmaDivSeq}
    For all $m,n\geq 0$, we have $(U_{mn})\subset (U_n)$.
\end{lemma}

\begin{proof}
    By induction on $m\geq 0$, and using Proposition \ref{PropUm+n}.
\end{proof}

We recall that two ideals $I$ and $J$ in $\mathscr{R}$ are said to be coprime if $I+J$ is equal to the whole ring $\mathscr{R}$. For $\mathscr{R}=\mathbb{Z}$, it is known that if $P$ and $Q$ are coprime, then $U$ satisfies interesting and strong properties. Now, we study the properties a Lucas sequence may satisfy when $(P)+(Q)=\mathscr{R}$.

\begin{lemma}\label{LemmaIdealsGen}
    Let $x,y,z\in \mathscr{R}$. If $(y)+(z)=\mathscr{R}$, then $(xy)+(z) = (x)+(z)$.
\end{lemma}

\begin{proof}
    One inclusion is trivial. Let $\alpha\in (x)+(z)$. Then, $\alpha=ax+bz$ for some $a,b\in \mathscr{R}$. But $\mathscr{R}=(y)+(z)$, so $a=uy+vz$ for some $u,v\in \mathscr{R}$. Hence $\alpha = uxy + (b+xv)z \in (xy)+(z)$.
\end{proof}

\begin{lemma}\label{LemmaQUn=R}
    If $(P)+(Q)=\mathscr{R}$, then $(Q)+(U_n)=\mathscr{R}$ for all $n\geq 1$.
\end{lemma}

\begin{proof}
    We proceed by induction on $n\geq 1$. If $n=1$, then $(Q)+(1)=\mathscr{R}$. Assume that the result is true for some $n\geq 1$. We have
    \begin{align*}
        (Q)+(U_{n+1}) &= (Q) + (PU_n-QU_{n-1}) \\
        &= (Q) +(PU_n) \\
        &= (Q)+(P) \\
        &=\mathscr{R}
    \end{align*}
    where we used Lemma \ref{LemmaIdealsGen} with $x=P$, $y=U_n$ and $z=Q$ in the second-to-last step and the assumption $(P)+(Q)=\mathscr{R}$ in the last step. The result follows by induction.
\end{proof}

\begin{proposition}\label{PropStrongDivSeq}
    We have $(P)+(Q)=\mathscr{R}$ if and only if $(U_m)+(U_n)=(U_{(m,n)})$ for all $m,n\geq 0$.
\end{proposition}

\begin{proof}
    For the ``if'' part, $\mathscr{R}=(U_2)+(U_3) = (P)+(P^2-Q)=(P)+(Q)$. For the converse, we start by showing the result holds for $m=n+1$ by induction on $n\geq 0$. The case $n=0$ is trivial, thus assume the result is valid for some integer $n\geq 0$. We have
    \[
    (U_{n+2})+(U_{n+1}) = (PU_{n+1}-QU_n)+(U_{n+1}) = (QU_n)+(U_{n+1}).
    \]
    By Lemma \ref{LemmaQUn=R}, we have $\mathscr{R}=(Q)+(U_{n+1})$. Thus, by Lemma \ref{LemmaIdealsGen} with $x=U_n$, we obtain
    \[
    (U_{n+2})+(U_{n+1})=(U_n)+(U_{n+1}),
    \]
    and we conclude by the induction hypothesis. Next, let $m\geq 0$ be an integer. Since $(m,n)=(n,m)$, we may assume without loss of generality that $m\geq n$ and write $m=qn+r$, where $q\geq 0$ and $0\leq r < n$. By Proposition \ref{PropUm+n}, we have
    \[
    (U_m) = (U_{r+1}U_{nq}-QU_rU_{nq-1}) \subset (U_{nq}) + (QU_rU_{nq-1}).
    \]
    We already know that $(U_{nq})+(U_{nq-1})=\mathscr{R}$, so that
    \[
    (U_{nq}) + (QU_rU_{nq-1}) = (U_{nq}) + (QU_r)
    \]
    by Lemma \ref{LemmaIdealsGen}. Moreover, since $(U_{nq})+(Q) = \mathscr{R}$ by Lemma \ref{LemmaQUn=R}, we obtain
    \[
    (U_{nq}) + (QU_rU_{nq-1}) = (U_{nq}) + (U_r)
    \]
    by Lemma \ref{LemmaIdealsGen}. By Lemma \ref{LemmaDivSeq}, we have $(U_{nq})\subset (U_n)$ and thus
    \[
    (U_m)+(U_n) \subset (U_{nq}) + (U_r)+(U_n) = (U_n)+(U_r).
    \]
    We obtain $(U_m)+(U_n) \subset (U_{(m,n)})$ by applying the Euclidean algorithm to this method. The other inclusion follows trivially from Lemma \ref{LemmaDivSeq}.
\end{proof}

A non-degenerate Lucas sequence $U=U(P,Q)$ with $(P)+(Q)=\mathscr{R}$ is called a {\it regular} Lucas sequence. If $\Delta\ne 0$, then $U$ is called a {\it $\Delta$-regular} Lucas sequence. 


Let $\mathfrak{a}\subsetneq \mathscr{R}$ be an ideal. We define the {\it rank of appearance} of $\mathfrak{a}$ in $U$, denoted by $\rho_U(\mathfrak{a})$ or $\rho$, to be the least integer $n\geq1$ such that $(U_n)\subset \mathfrak{a}$. If the rank does not exist, we write $\rho_U(\mathfrak{a})=+\infty$. Let $D:=(P)+(Q)$.

\begin{lemma}\label{LemmaRanknmidQ}
    If $D+\mathfrak{a}=\mathscr{R}$, then either $(Q) + \mathfrak{a} = \mathscr{R}$ or $\rho_U(\mathfrak{a})=+\infty$.
\end{lemma}

\begin{proof}
    If $\mathfrak{a}$ is the zero ideal, then $\rho_U(\mathfrak{a}) =+\infty$, as $U$ is non-degenerate. Assume $\mathfrak{a}\ne (0)$ and $(Q) + \mathfrak{a} \ne \mathscr{R}$. Then $\mathscr{R}$ is not a field and there exists a maximal ideal $\mathfrak{m}\subset \mathscr{R}$ such that $(Q)+\mathfrak{a} \subset \mathfrak{m}$. We have
    \[
    U_{n+2}-PU_{n+1} = -QU_n \in \mathfrak{m},
    \]
    for all $n\geq 0$. Thus, we have $U_{n+2} \equiv PU_{n+1} \equiv \cdots \equiv P^{n+1} \pmod{\mathfrak{m}}$, but $(Q) \subset \mathfrak{m}$ and $D+\mathfrak{m}=\mathscr{R}$ imply that $(P)\not \subset \mathfrak{m}$. Hence $U_n \not \in \mathfrak{m}$ for all $n\geq 2$. The case $n=1$ is easily verified. It follows that $U_n$ does not belong to $\mathfrak{a}$ for all $n\geq 1$ and that $\rho_U(\mathfrak{a})=+\infty$.
\end{proof}

Throughout the rest of this section, we let the letter $\mathfrak{p}$ denote a prime ideal in $\mathscr{R}$. Next, we study properties analogous to well-known prime divisibility properties in the $\mathscr{R}=\mathbb{Z}$ setting.

\begin{lemma}\label{LemmaUnUn+1ACCP}
    If $D+\mathfrak{p} = \mathscr{R}$, then $(U_{n+1})+(U_n)\not \subset \mathfrak{p}$ for all $n\geq 0$.
\end{lemma}

\begin{proof}
    Let $\rho=\rho_U(\mathfrak{p})$. If $\rho=+\infty$, there is nothing to do. If $\rho<+\infty$, we proceed by induction on $n\geq 0$. If $n=0$, then $(0)+(1)=\mathscr{R} \not \subset \mathfrak{p}$. Assume the result holds for some $n\geq 0$ and that
    \begin{equation}\label{eqAbsurdInc}
        (U_{n+2})+(U_{n+1}) = (QU_n) + (U_{n+1}) \subset \mathfrak{p},
    \end{equation}
    by contradiction. We find that $(QU_n)\subset \mathfrak{p}$. Since $\mathfrak{p}$ is a prime ideal, $QU_n\in \mathfrak{p}$ implies that $Q$ or $U_n$ belongs to $\mathfrak{p}$. However, since the rank of $\mathfrak{p}$ exists, we must have $(Q)\not \subset \mathfrak{p}$ by Lemma \ref{LemmaRanknmidQ}. Thus we conclude that $(U_n)\subset \mathfrak{p}$. By \eqref{eqAbsurdInc}, we have $(U_{n+1})\subset \mathfrak{p}$ and therefore
    \[
    (U_{n+1})+(U_n) \subset \mathfrak{p},
    \]
    contradicting the induction hypothesis. Hence $(U_{n+2})+(U_{n+1}) \not  \subset \mathfrak{p}$.
\end{proof}

\begin{proposition}\label{PropPdivUrhodivn}
    If $D+\mathfrak{p}=\mathscr{R}$ and the rank $\rho_U(\mathfrak{p})$ exists, then $(U_n)\subset \mathfrak{p}$ if and only if $\rho_U(\mathfrak{p})\mid n$.
\end{proposition}

\begin{proof}
    The ``if'' part follows from Lemma \ref{LemmaDivSeq}. Assume that $(U_n)\subset \mathfrak{p}$. By minimality of the rank, we must have $n\geq \rho:=\rho_U(\mathfrak{p})$. Hence $n=q\rho+r$ for some $q\geq 1$ and $0\leq r <\rho$. By Proposition \ref{PropUm+n}, we have
    \[
    U_n=U_{r+1}U_{q \rho } -Q U_r U_{q\rho -1}.
    \]
    By Lemma \ref{LemmaDivSeq}, this implies that $(Q U_r U_{q\rho -1}) \subset \mathfrak{p}$. Since $D+ \mathfrak{p}=\mathscr{R}$, we use Lemma \ref{LemmaUnUn+1ACCP} to find that $(U_{q \rho}) + (U_{q\rho -1}) \not \subset \mathfrak{p}$. Since $(U_{q \rho}) \subset \mathfrak{p}$, we obtain that $(U_{q\rho -1}) \not \subset \mathfrak{p}$. Thus, we have $Q U_r U_{q\rho -1} \in \mathfrak{p}$ and $Q U_{q\rho -1} \not \in \mathfrak{p}$. Since $\mathfrak{p}$ is a prime ideal, we find that $U_r\in \mathfrak{p}$, so $(U_r) \subset \mathfrak{p}$. But $r<\rho$ implies that $r=0$ and $n=q\rho$ by minimality of the rank.
\end{proof}

\subsection{Unique factorization domains}

Let $R$ be a unique factorization domain with characteristic $p$. Hence, irreducible elements $x\in R$ are prime elements, that we call primes. We denote both $x$ and the prime ideal $(x)$ by $\mathfrak{p}$ as a shorthand. We let $v_\mathfrak{p}$ denote the $\mathfrak{p}$-adic valuation, the function defined by
\[
v_\mathfrak{p}(x) = \max\{n\geq 1 : \mathfrak{p}^n \mid x\},
\]
for all non-zero $x\in R$, and $v_{p}(0)=-\infty$. Moreover, a valuation satisfies the following properties:
\[
v_\mathfrak{p}(xy)=v_\mathfrak{p}(x)+v_\mathfrak{p}(y) \quad \text{and} \quad v_\mathfrak{p}(x+y)\geq \min(v_\mathfrak{p}(x), v_\mathfrak{p}(y) ),
\]
for all $x,y\in R$. We recall that a valuation and its properties can be extended to the field of fractions of $R$ by $v_\mathfrak{p}(x/y) = v_\mathfrak{p}(x)-v_\mathfrak{p}(y)$ for all $x,y\in R$.

Let $U$ be a regular Lucas sequence with parameters $P,Q\in R$, i.e., it satisfies $(P)+(Q)=R$. By Proposition \ref{PropPdivUrhodivn},  we know that if $\rho_U(\mathfrak{p})$ exists, then $\mathfrak{p}$ divides a term $U_n$ if and only if $\rho_U(\mathfrak{p})$ divides $n$. The aim of this subsection is to describe the behavior of $v_\mathfrak{p}(U_n)$ for all integers $n\geq 1$ divisible by $\rho_U(\mathfrak{p})$.

\begin{lemma}\label{LemmaUpn=}
    Assume $p>0$. For all $i,n\geq 0$, we have $U_{p^i n}=\Delta^{\frac{p^i-1}{2}} U_n^{p^i}$. 
\end{lemma}

\begin{proof}
    Since $U$ is regular, it is non-degenerate. By Lemma \ref{LemmaDegenerate}, and because $p>0$, we have $a-b=\sqrt{\Delta}\ne 0$ and
    \[
    U_{p^in} = \frac{a^{p^in}-b^{p^in}}{a-b} = \frac{(a^n-b^n)^{p^i}}{a-b} = (a-b)^{p^i-1} U_n^{p^i} =\Delta^{\frac{p^i-1}{2}} U_n^{p^{i}}.
    \]
\end{proof}

\begin{lemma}\label{LemmaRho=p}
    If $p>0$ and $\rho_U(\mathfrak{p})$ exists, then $\mathfrak{p}\mid \Delta$ if and only if $\rho_U(\mathfrak{p}) = p$.
\end{lemma}

\begin{proof}
    By Lemma \ref{LemmaUpn=}, we have $U_p = \Delta^{(p-1)/2}$. (Note that $U_2=P$ and $\Delta=P^2$ in characteristic $2$.) Thus, all primes of rank $p$ divide $\Delta$. For the converse, we find that $\rho_U(\mathfrak{p}) \mid p$ by Proposition \ref{PropPdivUrhodivn}. Hence $\rho_U(\mathfrak{p})=p$ since $U_1=1$.
\end{proof}

\begin{theorem}\label{THMvaluationLucas}
    Assume $p>0$ and $\rho:=\rho_U(\mathfrak{p})$ exists. Then
    \[
    v_\mathfrak{p}(U_{\rho n}) = p^{v_p(n)} v_\mathfrak{p}(U_\rho) + \frac{(p^{v_p(n)}-1)v_\mathfrak{p}(\Delta)}{2}.
    \]
\end{theorem}

\begin{proof}
    Write $n=\lambda p^u$ for some integers $\lambda\geq 1$, $p\nmid \lambda$, and $u\geq 0$. Since $U$ is a regular sequence, by Proposition \ref{PropStrongDivSeq}, we have $(U_n, U_{\rho p^{u+1}}) = U_{(n, \rho p^{u+1})} = U_{\rho p^u}$. The result follows by taking the $\mathfrak{p}$-adic valuation and using Lemma \ref{LemmaUpn=}.
    
\end{proof}

The next theorem deals with the case of a unique factorization domain and a polynomial ring $R$ of characteristic zero. The key point is that $R$ is a polynomial ring and we can use the notion of polynomial degree. In particular, any irreducible element $\mathfrak{p}$ is an irreducible polynomial with positive degree. Note that the theorem is not valid otherwise, as \cite[Sect. 2.4, Lemma 11]{BaWi} shows for $R=\mathbb{Z}$.

\begin{theorem}\label{ThmValp=0}
    Assume that $R$ is a polynomial ring with $p=0$, and that $\rho_U(\mathfrak{p})$ exists. Then, we have $v_\mathfrak{p}(U_{\rho n}) = v_\mathfrak{p}(U_\rho)$ for all $n\geq 1$.
\end{theorem}

\begin{proof}
    Let $V=(V_n)_{n\geq 0} \subset R$ be the sequence defined by $V_n=a^n+b^n$ for all $n\geq 0$. It is called a {\it companion Lucas sequence}. We may prove the following formulas in any integral domain with zero characteristic:
     \begin{equation}\label{EdForm1}
        2^{n-1} U_{kn} = \sum_{i=0}^{\floor{(n-1)/2}} \binom{n}{2i+1} \Delta^i U_k^{2i+1} V_k^{n-2i-1},
    \end{equation}
    for all $m,n\geq 0$, and $V_n^2-\Delta U_n^2 = 4Q^n$ for all $n\geq 0$. These are classical formulas in the theory of Lucas sequences and proofs can be found in the book of Ballot and Williams \cite[Sect. 2.2, equations (2.28) and (2.8)]{BaWi} for $\mathbb{Z}$. This classical proof is valid in any commutative ring of characteristic different from $2$. Now, the formula $V_n^2-\Delta U_n^2 = 4Q^n$ shows that $(U_n,V_n)$ divides $4Q^n$ for all $n\geq 0$. It follows that $\mathfrak{p}$ is not a divisor of $V_\rho$ as $\mathfrak{p}\nmid Q$ by Lemma \ref{LemmaRanknmidQ}, and $\mathfrak{p}$ has positive degree. Putting $k=\rho$ in \eqref{EdForm1}, we obtain
    \[
    2^{n-1} U_{\rho n} \equiv n U_\rho V_\rho^{n-1} \pmod{U_\rho^2}.
    \]
    If $\lambda=v_\mathfrak{p}(U_\rho)$, then this implies
    \[
    2^{n-1} U_{\rho n} \equiv n U_\rho V_\rho^{n-1} \pmod{\mathfrak{p}^{\lambda+1}},
    \]
    but $\mathfrak{p}^\lambda \| n U_\rho V_\rho^{n-1}$ and we find that $v_\mathfrak{p}(U_{\rho n}) = \lambda$.
\end{proof}

\section{The main theorem}\label{Section3}

Let $A$ be a unique factorization domain with characteristic $p$. We consider the case of the polynomial ring $R=A[T]$. We use the letter $R$ as in the previous section because $R$ is a unique factorization domain. Let $K$ denote the field of fraction of $R$. In this section, we prove our main theorem on primitive prime divisors of Lucas sequences with parameters $P,Q\in R$. We assume that $P$ and $Q$ are non-zero, coprime and that one of $P$ or $Q$ has positive polynomial degree. These assumptions ensure that the characteristic polynomial of $U$ has distinct roots. Indeed, the roots are equal if the discriminant $\Delta=P^2-4Q$ is zero, which can not happen if $p=2$, since $P$ is non-zero. If $p\ne 2$, then $2\deg(P) = \deg(Q) \geq 1$, and $P$ and $Q$ can not be coprime. We follow a method of Yabuta \cite{Ya}.

Let $(Q_n)_{n\geq 0}$ be the sequence defined by $Q_0=Q_1=1$ and for all $n\geq 2$ by $Q_n:=Q_n(a,b)=\Phi_n(a,b)$, where $\Phi_n$ is the $n$-th homogeneous cyclotomic polynomial defined by
\[
\Phi_n(X,Y) = \prod_{\substack{1\leq k\leq n \\ (k,n)=1}} (X-\zeta_n^kY),
\]
for all $n\geq 1$. We have the following well-known identity:
\[
X^n-Y^n = \prod_{d\mid n} \Phi_d(X,Y) = (X-Y) \prod_{\substack{d\mid n \\ d>1}} \Phi_d(X,Y),
\]
which, applied to the Lucas sequence $U$, yields
\begin{equation}\label{U_n=ProdQ_n}
    U_n = \frac{a^n-b^n}{a-b} = \prod_{\substack{d\mid n \\ d>1}} \Phi_d(a,b) =\prod_{d\mid n} Q_d.
\end{equation}
We want to study the prime divisors of $U_n$ by looking at those of $Q_d$ for all $d\mid n$. However, we first need to check that $Q_n$ belongs to $R$ for all $n\geq 1$. It is true for $n=1$ and $n=2$ since $Q_1=1$ and $Q_2=P$. For $n\geq 3$, we first note that $\varphi(n)$ is even, where $\varphi$ is Euler's totient function. Moreover, following the method used in \cite[p.\,89]{BaWi}, we may prove that $Q_n(a,b)=Q_n(b,a)$. It follows that $Q_n\in K=\mathrm{Frac}(R)$. Indeed, let $L=K(a)$ be the splitting field of the characteristic polynomial of $U$, that is, $X^2-PX+Q$. It is a Galois extension of $K$ of degree $1$ or $2$. Hence $Q_n\in K$ if and only if $\sigma(Q_n)=Q_n$ for the non-trivial automorphism $\sigma$ of $L/K$, when it exists. If $L$ is a quadratic extension of $K$, then 
\[
\sigma(Q_n(a,b))= Q_n(\sigma(a),\sigma(b))=Q_n(b,a) = Q_n(a,b),
\]
because $\sigma$ sends $a$ to $b$. Thus $Q_n\in K$. To prove that $Q_n$ is integral, we show that $v_\mathfrak{p}(Q_n)\geq 0$ for all irreducible elements $\mathfrak{p}\in R$. By $\eqref{U_n=ProdQ_n}$, we have
\[
v_\mathfrak{p}(U_n)=\sum_{d\mid n} v_\mathfrak{p}(Q_d),
\]
and, using the M\" obius inversion formula, we obtain
\[{}
v_\mathfrak{p}(Q_n) = \sum_{d\mid n} \mu(n/d)v_\mathfrak{p}(U_d).
\]
In prime characteristic, Lemma \ref{LemmaUpn=} shows that it suffices to determine $v_\mathfrak{p}(Q_n)$ for all $n\geq 3$ not divisible by $p$. Indeed, if $n=mp^i$, we have
\[
U_n = \Delta^{\frac{p^i-1}{2}} U_m^{p^i},
\]
Therefore, since $\Delta\in R$, we have $Q_n\in R$ if and only if $Q_d \in R$ for all $d\mid m$ by \eqref{U_n=ProdQ_n}.
For simplification, we write $p\nmid n$ even in characteristic zero since it is equivalent to $n\geq 1$. If $\rho_U(\mathfrak{p})\nmid n$, then $v_\mathfrak{p}(Q_n)=0$ and we are done. Assume $\rho =\rho_U(\mathfrak{p})\mid n$. By Theorem \ref{THMvaluationLucas} if $p>0$ and Theorem \ref{ThmValp=0} if $p=0$, we have
\[
v_\mathfrak{p}(Q_n) = \sum_{d\mid n,\ \rho \mid d} \mu(n/d)v_\mathfrak{p}(U_d) =  v_\mathfrak{p}(U_\rho) \sum_{d' \mid \frac{n}{\rho}} \mu(n/\rho d'),
\]
for all $n\geq 3$, $p\nmid n$. The last sum is well-known to be equal to $1$ if $n/\rho=1$ and $0$ if $n/\rho>1$. Hence we just proved the following lemma:
\begin{lemma}\label{LemmaQnUn}
    We have $Q_n\in R$ for all $n\geq 1$. Moreover, if $p\nmid n$, then $\mathfrak{p}\mid Q_n$ if and only if $n=\rho_U(\mathfrak{p})$.
\end{lemma}
Suppose $p\nmid n$. It follows from Lemma \ref{LemmaQnUn} that $U_n$ has no primitive prime divisor if and only if $Q_n$ is a constant in $R$, i.e., $Q_n\in A$. Now, assuming $n\geq 3$, we have $0<k<n/2$ only if $n-k>n/2$. Therefore, using the identities $a^2+b^2=P^2-2Q$ and $ab=Q$, we may write
\begin{align}
    Q_n &= \prod_{\substack{0<k<n/2 \\ (k,n)=1}} (a-\zeta_n^kb)(a-\zeta_n^{n-k}b) \nonumber \\ 
    &=\prod_{\substack{0<k<n/2 \\ (k,n)=1}} (a^2+b^2-(\zeta_n^k+\zeta_n^{n-k})ab ) \nonumber \\
    &= \prod_{\substack{0<k<n/2 \\ (k,n)=1}} (P^2-\theta_kQ), \label{QnProdFormula}
\end{align}
where $\theta_k=2+\zeta_n^k+\zeta_n^{-k}$. We work on the above product formula for $Q_n$ to prove our main theorem. Indeed, we prove that for all but finitely many $n\geq 3$, this product has a non-constant factor.

\begin{lemma}\label{Lemmaa+a-1}
    Let $\alpha,\beta\in R^\times$. We have $\alpha+\alpha^{-1} = \beta+\beta^{-1}$ if and only if $\alpha=\beta$ or $\alpha=\beta^{-1}$.
\end{lemma}

\begin{proof}
    The ``if'' part is trivial. For the converse, note that the polynomial $f(x)=x^2-(\alpha+\alpha^{-1})x+1$ annihilates both $\alpha$ and $\beta$ because $\alpha+\alpha^{-1}=\beta+\beta^{-1}$. We have
    \begin{align*}
        f(\alpha)=f(\beta) &\iff \alpha^2-\alpha(\alpha+\alpha^{-1})+1=\beta^2-\beta(\alpha+\alpha^{-1})+1 \\
        &\iff \alpha^2-\beta^2 = (\alpha-\beta)(\alpha+\alpha^{-1}).
    \end{align*}
    If $\alpha\ne \beta$, we may divide by $\alpha-\beta$ and we obtain $\alpha+\beta=\alpha+\alpha^{-1}$, or equivalently, $\alpha=\beta^{-1}$.
\end{proof}

\begin{theorem}
    Let $A$ be a unique factorization domain and $R=A[T]$. Let $U$ be a regular Lucas sequence with non-zero parameters $P,Q\in R$ such that one of $P$ and $Q$ has positive degree. Assume $n\geq 2$ and $p\nmid n$. Then $U_n$ has no primitive prime divisor if and only if there exists a non-zero constant $\lambda \in A$ such that one of the following holds:
    \begin{enumerate}[label=(\arabic*)]
        \item $n=2$, $P=\lambda$ and $\deg{(Q)}\geq 1$,

        \item $n=3$ and $P^2=Q+\lambda$,

        \item $n=4$ and $P^2=2Q+\lambda$,

        \item $n=6$ and $P^2=3Q+\lambda$.
    \end{enumerate}
    In particular, we see that at most one of the above can hold.
\end{theorem}

\begin{proof}
    The case $n=2$ yields $U_2=P=\lambda$. This is clearly sufficient. For $n\geq 3$, it suffices to check whether $Q_n$ belongs to $A$ by Lemma \ref{LemmaQnUn}. By \eqref{QnProdFormula}, we want to show that at least one of the $\varphi(n)/2$ factors $P^2-\theta_kQ$ has positive degree, where $0<k<n/2$ and $(k,n)=1$. If $\deg(P^2)\ne \deg(Q)$, then $\deg(P^2-\theta_kQ)$ is equal to the maximum of $\deg(P^2)$ and $\deg(Q)$, which is positive. Thus $Q_n\not \in A$ and $U_n$ has a primitive divisor by Lemma \ref{LemmaQnUn}. Assume $2\deg{(P)}= \deg{(Q)}$ and $\varphi(n)>2$. Then $Q_n$ has a least two factors of the form $P^2-\theta_k Q$. By contradiction, assume that there exist $i$ and $j$, $i\ne j$, such that both
    \[
    P^2-\theta_iQ \quad \text{and} \quad P^2-\theta_jQ
    \]
    are constants, say $\lambda_i$ and $\lambda_j$. We obtain $Q(\theta_j-\theta_i)=\lambda_i-\lambda_j$. However, we have $\deg(\lambda_i-\lambda_j)\leq 0$ and $\deg(Q)\geq 1$, since $P$ and $Q$ are not both constants by hypothesis. It follows that $\lambda_i=\lambda_j$ and $\theta_i=\theta_j$. The latter is equivalent to $\zeta_n^i+\zeta_n^{-i} = \zeta_n^j+\zeta_n^{-j}$. By Lemma \ref{Lemmaa+a-1}, we either have $\zeta_n^i=\zeta_n^j$ or $\zeta_n^i=\zeta_n^{-j}$. Hence $n\mid i-j$ or $n\mid i+j$, but $0<i,j<n/2$ implies that $i=j$, a contradiction. It follows that one and only one factor of $Q_n$ in \eqref{QnProdFormula} can be constant. Since $\varphi(n)>2$, $\varphi(n)$ is even and $Q_n$ has $\varphi(n)/2$ factors in $\eqref{QnProdFormula}$, we find that $\deg(Q_n)\geq 1$. The remaining cases for $2\deg{(P)}= \deg{(Q)}$ are integers $n\geq 3$ such that $\varphi(n)=2$, that is, $n=3,4$ and $6$. Note that we have the following:
    \[
    U_3=P^2-Q, \quad U_4=U_2(P^2-2Q), \quad \text{and} \quad U_6=U_2U_3(P^2-3Q).
    \]
    By \eqref{U_n=ProdQ_n}, we obtain that $Q_3=P^2-Q$, $Q_4=P^2-2Q$, and $Q_6=P^2-3Q$. The result follows.
\end{proof}

In conclusion, we obtained that the greatest integer $n\geq 1$, $p\nmid n$, for which $U_n$ has no primitive divisor is at most equal to $6$. For each $1\leq n \leq 6$, we found the conditions on $P$ and $Q$ for $U_n$ not to have
a primitive divisor. Note that the method we used is enough to obtain a similar result for sequences of the form $(a^n-b^n)_{n\geq 0}$, with $a,b\in R$, and for Lehmer sequences in $R$.

\end{document}